\documentclass{amsart}
\usepackage{latexsym,amsmath,amsthm,amsfonts,eucal,amscd,amssymb,enumerate}

\newtheorem{thm}{Theorem}

\newtheorem{lem}[thm]{Lemma}

\theoremstyle{definition}
\newtheorem*{defn}{Definition}
\newtheorem*{rem}{Remark}
\newtheorem{ex}[thm]{Example}

\def\A{\mathbb{A}}
\def\zz{\mathbb{Z}}

\def\qq{\mathbb{Q}}
\def\pp{\mathbb{P}}

\DeclareMathOperator{\Aut}{Aut}
\DeclareMathOperator{\Res}{Res}
\DeclareMathOperator{\PGL}{PGL}

\DeclareMathOperator{\id}{id}
\DeclareMathOperator{\ch}{char}
\DeclareMathOperator{\N}{N}

\begin{document}

\title{Explicit Descriptions of Quadratic Maps on $\pp^1$ defined over a field $K$}

\author[Manes]{Michelle Manes}
\address{
Department of Mathematics \\
University of Hawaii \\
Honolulu, HI  96822 \\ USA
}
\email{mmanes@math.hawaii.edu}

\author[Yasufuku]{Yu Yasufuku}
\address{
City University of New York Graduate Center\\
New York, NY 10016\\ USA
}
\email{yasufuku@post.harvard.edu}

\subjclass[2010]{
37P05,
37P45
(primary);
14G05
(secondary).}
\keywords{arithmetic dynamics, quadratic rational maps, normal forms}

\begin{abstract}
We describe an explicit parameter space for the set of all quadratic rational maps on $\pp^1$ defined over a field $K$, up to conjugacy over $K$.
 \end{abstract}

\date{\today}

\maketitle

\section{Introduction}
Let $\phi:\pp^1 \to \pp^1$ be a morphism of degree~$d$ defined over a field $K$.  In other words, $\phi(z) = P(z)/Q(z)$ with $P, Q \in K[z]$ having no common root in $\overline K$,  and $\max\{\deg P, \deg Q\} = d$.  Such rational maps are the fundamental objects of study in one-dimensional arithmetic dynamics.  (Because we are working in one dimension, rational maps and morphisms exactly coincide, so we use the terms interchangeably.) 
Performing the same change of coordinates on both the domain and target spaces preserves all dynamical behavior, so we are usually interested in studying conjugacy classes rather than individual rational maps.

Generalizing
work by Milnor~\cite{milnrat}, Silverman~\cite{mdgit} proved that the moduli space of degree-$d$ rational maps up to conjugacy, which we denote $\mathcal M_d$, exists as an affine integral scheme over $\zz$ and that
$\mathcal M_2$ is isomorphic to $\A^2_{\zz}$.   As explained in Section~\ref{prelim}, if $\phi:\pp^1\to\pp^1$ has degree~$2$, and
if we let $\lambda_1, \lambda_2, \lambda_3$ be the multipliers of the three fixed points of $\phi$ (counted with multiplicity), then the first two symmetric functions of these multipliers form natural coordinates for $\mathcal M_2$:
\begin{equation}\label{m2coords}
\mathcal M_2  =\{ (\sigma_1, \sigma_2)\} \text{ where }
\sigma_1 = \lambda_1 + \lambda_2 + \lambda_3, \text{ and }
\sigma_2 =  \lambda_1 \lambda_2 + \lambda_1 \lambda_3 +\lambda_2 \lambda_3.
\end{equation}

In~\cite{fodfom} Silverman shows that in the case of polynomial maps and maps of even degree, the field of moduli for a rational map is always a field of definition.  In particular, a $K$-rational point in $\mathcal M_2$ corresponds to a conjugacy class of quadratic rational maps $[\psi]$, and some map $\phi \in [\psi]$ must have coefficients in the field $K$.  However, it was not clear from any previous work how to explicitly find such a map, given a $K$-rational point in the moduli space.

Furthermore, each family $[\psi]$ describes a conjugacy class of maps, but only up to conjugacy over an algebraically closed field $\overline K$.  Much recent work in arithmetic dynamics has focused on quadratic polynomials.  These results invariably use the normal form $z^2+c$ and the fact that it gives a complete description of quadratic polynomials up to $K$-conjugacy.  (See, for example,~\cite{PreImages}, \cite{HutzIngAlg}, and~\cite{IngramHts}.)  Attempts to extend these results to arbitrary quadratic rational maps would benefit from a complete description of the $K$-conjugacy classes of such maps.  Our  result provides such a description.

\begin{thm}\label{mainthm}
Let $K$ be a field with characteristic different from $2$ and $3$.
Let $\psi(z) \in K(z)$ have degree~$2$, and let $\lambda_1, \lambda_2, \lambda_3 \in \overline K$ be the multipliers of the fixed points of $\psi$ \textup(counted with multiplicity\textup).
\begin{enumerate}[\textup(a\textup)]

\item\label{noaut1}
If the multipliers are distinct or if exactly two multipliers are~$1$, then $\psi(z)$ is conjugate over $K$ to the map
\[
\phi(z) = \frac{2 z^2 + (2-\sigma_1) z + (2-\sigma_1)}{-z^2 + (2+\sigma_1)z + 2-\sigma_1-\sigma_2} \in K(z),
\]
where $\sigma_1 $ and $\sigma_2$ are the first two symmetric functions of the multipliers.  Furthermore, no two distinct maps of this form are conjugate to each other over~$\overline K$.

\item\label{c2aut1}\label{c2aut2}
If $\lambda_1 = \lambda_2 \neq 1$ and $\lambda_3 \neq \lambda_1$ or if $\lambda_1=\lambda_2=\lambda_3=1$, then $\psi$ is conjugate over $K$ to a  map of the form
\begin{equation*}
\phi_{k,b}(z) = kz + \frac b z
\end{equation*}
with $k\in K\smallsetminus\{0,-1/2\}$ \textup(in fact, $k = \frac{\lambda_1+1}{2}$\textup),  and $b \in K^*$.
Furthermore, two such maps $\phi_{k,b}$ and $\phi_{k',b'}$ are conjugate over $\overline K$ if and only if $k = k'$; they are conjugate over $K$ if in addition $b/b' \in \left(K^*\right)^2$.

\item\label{s3aut}
If  $\lambda_1=\lambda_2=\lambda_3=-2$, then $\psi$ is conjugate over $K$ to a  map of the form
\begin{equation*}
\theta_{d,k}(z) = \frac{k z^2 - 2 d z +dk}{z^2 - 2 k z + d},
\qquad \text{with } k \in K, d \in K^*, \text{ and } k^2\neq d.
\end{equation*}
 All such maps are conjugate over $\overline K$.  Furthermore, $\theta_{d,k}(z)$ and $\theta_{d',k'}(z)$ are conjugate over $K$ if and only if
  \begin{align*}
  d'&=b^2 d, \text{ and }\\
k' &\in \left\{
 \frac {bd}{k},  \frac{b \left(d^2 \gamma ^3+3 d k \gamma ^2+3 d \gamma +k\right)}{d k
   \gamma ^3+3 d \gamma ^2+3 k \gamma +1}
\right\}
\end{align*}
for some $\gamma \in K$ and $b\in K^*$.
\end{enumerate}
Each quadratic rational map $\phi(z) \in K(z)$ must fall into exactly one of the cases above, so this gives a complete description of the $K$-conjugacy classes of such maps.
\end{thm}

\begin{rem}
As will be proved in Lemma \ref{lem:s3}, if $\psi$ satisfies the hypothesis of case \eqref{s3aut}, then it has a two-cycle defined over $K$ if and only if it is conjugate over $K$ to $\theta_{d,k}$ with $d\in (K^*)^2$.  In such a case, $\psi$ is also conjugate over $K$ to
\begin{equation*}
\theta_t(z) = t/z^2,  \quad \text{with } t = \frac{-\sqrt d + k}{\sqrt d +k}.
\end{equation*}
Since $k^2\neq d$, we see that $t$ is a well-defined element of $K^*$.  With this alternative description, $\theta_t(z)$ and $\theta_{t'}(z)$ are conjugate over $K$ if and only if $t/t' \in (K^*)^3$ or $t t' \in (K^*)^3$.
\end{rem}

The theorem is proved via several lemmas, each tackling a specific case of quadratic rational maps based on their automorphism groups.

\subsection*{Acknowledgements}
The authors thank Rob Benedetto for initially posing the question and for suggesting the Remark after Lemma~\ref{lem:s3}.  We also thank Xander Faber, Joe Silverman, and the referee for their helpful and insightful comments, all of which greatly improved the paper.

\section{Preliminaries}\label{prelim}
Throughout, we take $K$ to be a field with characteristic different from $2$ and $3$.
 Let $\phi: \pp^1 \rightarrow  \pp ^1$ be a rational map defined over~$K$.  That is, let  $\phi(z) = P(z)/Q(z)$ with $P, Q \in K[z]$ having no common root in $\overline K$, and $\deg \phi = \max\{ \deg P, \deg Q\}$.
 
 \begin{defn}
Two rational maps $\phi, \psi \in K(z)$ are \emph{conjugate} if there is some $h\in \PGL_2(\overline K)$ such that
\[
    \phi^h \overset{\text{def}}{=} h^{-1}\circ  \phi \circ h= \psi,
\]
 and they are \emph{conjugate over $K$} if we can take $h \in \PGL_2(K)$.
\end{defn}

If $\deg(\phi) = d$, then $\phi$ has $d+1$ fixed points, counted with proper multiplicity.   If $P$ is a finite fixed point of $\phi$, the multiplier at $P$ is defined to be $\phi'(P)$.  For any $h\in \PGL_2$, $h^{-1}(P)$ is a fixed point of $\phi^h$.  Applying the chain rule shows that the multiplier of $\phi$ at $P$ is equal to the multiplier of $\phi^h$ at $h^{-1}(P)$ (as long as both are finite points).  Therefore, we can speak of the set of multipliers of a conjugacy class of rational maps, and we can also extend the definition of the multiplier to the point at infinity.

A finite fixed point has multiplier equal to~$1$ if and only if it is a multiple root of the polynomial $P(z) - z Q(z)$ (see~\cite[Theorem 4.6]{ads}), and this extends to a fixed point at infinity via conjugation in the obvious way.  So for a quadratic rational map, either no multiplier is~$1$ (if the three fixed points are distinct), or two multipliers are~$1$ (corresponding to a fixed point of multiplicity two), or all three multipliers are~$1$ (corresponding to a fixed point of multiplicity three).

Let $\lambda_1, \cdots, \lambda_{d+1}$ be the fixed point multipliers for a rational map $\phi$ of degree~$d$.  As long as none of the $\lambda_i$ are~$1$,  by~\cite[Theorem 1.14]{ads}, we have
\begin{equation}\label{fpident}
\sum_{i=1}^{d+1}
1/(1-\lambda_i)  = 1.
\end{equation}
We conclude that for quadratic rational maps, all three multipliers can be equal only if they are all equal to~$1$, or if they are all equal to~$-2$; and if two multipliers are equal, they cannot be $-1$.  It is now clear that the three categories in Theorem~\ref{mainthm} exhaust all possibilities for quadratic rational maps, so by proving the theorem we will have a complete description of all $K$-conjugacy classes of such maps.

Usually, $\phi^h \neq \phi$ as rational maps, but this is not always the case.  For example, the map $\phi(z) = 2z+ 5/z$ has a nontrivial $\PGL_2$ automorphism $h(z) = -z$.

\begin{defn}
The  \emph{automorphism group} of $\phi\in K(z)$  is
$$\Aut(\phi) = \{f\in \PGL_2\left(\overline K \right) | \phi^f =\phi\}.$$
\end{defn}

If $f \in \Aut(\phi)$, then $h^{-1}\circ f \circ h \in \Aut(\phi^h)$ for any $h\in \PGL_2(\overline K)$, so  conjugate maps have isomorphic automorphism groups.
 If $\deg\phi \geq 2$, then $\Aut(\phi)$ must be a finite subgroup of $\PGL_2\left(\overline K \right)$; furthermore,  if two maps defined over a field $K$ are  conjugate, then they must be  conjugate over $K$ unless the maps have a nontrivial automorphism group (this follows from~\cite[Proposition~$7.2$]{fodfom} or~\cite[Proposition $4.73$]{ads}).  This last fact is essential in our classification of $K$-conjugacy classes.

 In the case of quadratic rational maps, Milnor showed in~\cite[Section~$5$]{milnrat} that
 \begin{itemize}
 \item
 the automorphism group is trivial if and only if all three multipliers are distinct or exactly two of the multipliers equal~$1$ (case~\eqref{noaut1} of Theorem~\ref{mainthm});  
 \item
 the automorphism group is cyclic of order~2 (we will use the notation $\mathfrak C_2$) if and only if two multipliers are equal but are not~$1$ or all three multipliers are~$1$ (case~\eqref{c2aut2} of Theorem~\ref{mainthm}); and
 \item
the automorphism group is isomorphic to the symmetric group $\mathfrak S_3$ if and only if all three multipliers are~$-2$ (case~\eqref{s3aut} of Theorem~\ref{mainthm}).  
\end{itemize}
Hence, we proceed by considering each possible automorphism group.

Finally, if $[\phi]$ corresponds to a $K$-rational point in the moduli space $\mathcal M_2$, then the symmetric functions of the multipliers satisfy $\sigma_1, \sigma_2 \in K$.  It follows from equation~\eqref{fpident} that $\sigma_3=\sigma_1-2$.  So in fact the three multipliers are roots of a monic cubic polynomial with coefficients $K$.  We will frequently use the fact that a quadratic rational map is completely determined up to $\overline K$-conjugacy by the three multipliers $ \lambda_1, \lambda_2, \lambda_3$, or equivalently by the first two symmetric functions on these multipliers $\sigma_1$ and $\sigma_2$.  See~\cite[Lemma~$3.1$]{milnrat} for details.

Using the coordinates $(\sigma_1, \sigma_2)$, Milnor provides a description of the symmetry locus in $\mathcal M_2$.  The maps with nontrivial automorphism groups lie on the cuspidal cubic defined by
\begin{equation}\label{symmlocus}
-2\sigma_1^3 - \sigma_1^2\sigma_2 + \sigma_1^2 + 8 \sigma_1\sigma_2 
+ 4\sigma_2^2-12\sigma_1 - 12\sigma +36 = 0.
\end{equation}
(This equation can be derived from the parameterization given in~\cite[Corollary 5.3]{milnrat}.)

\section{The case $\Aut(\phi) = \id$}
Let $[\psi] \in \mathcal M_2$ be a conjugacy class of quadratic rational maps with trivial automorphism group.  Then  either all three multipliers are distinct, or exactly two multipliers are~$1$.  From the remarks in Section~\ref{prelim}, we know that if $[\psi]$ corresponds to a $K$-rational point in $\mathcal M_2$, then we can find a map $\phi\in [\psi]$ defined over~$K$, and that there is only one $K$ conjugacy class for $[\psi]$.  Hence, it suffices to find a single map defined over $K$ for each $K$-rational point in $\mathcal M_2$.

\begin{lem}\label{alldif}
Let $[\psi]\in \mathcal M_2(K)$ correspond to a $K$-rational point $(\sigma_1, \sigma_2)$ in the moduli space of degree-$2$ rational maps, and let $\psi \in [\psi]$ be any representative.  If $\Aut(\psi)$ is trivial,  then  there is a unique map $\phi(z)$ in $[\psi]$ of the form
\begin{equation}\label{phidefalldiffeqn}
\phi(z)=\frac{P(z)}{Q(z)} = \frac{2 z^2 + (2-\sigma_1) z + (2-\sigma_1)}{-z^2 + (2+\sigma_1)z + 2-\sigma_1-\sigma_2},
\end{equation}
which is necessarily defined over $K$.
\end{lem}

\begin{proof}
Define $\phi$, $P$, and $Q$ as in~\eqref{phidefalldiffeqn}. 
Recall that $\Aut(\psi)$ is trivial if and only if the three fixed point multipliers are distinct or exactly two of the multipliers are~$1$.  We treat these two cases separately.

First, assume that exactly two of the multipliers for fixed points of $\psi$ are~$1$, and let the third multiplier be $\lambda\neq 1$.  In this case, $\sigma_1 = 2+\lambda$, $\sigma_2=2\lambda+1$, and
\[
\phi(z) =  
\frac{2 z^2-\lambda z-\lambda}{-z^2+(\lambda +4) z-3 \lambda-1}.
\]
This map has a double fixed point at $z=1$ with multiplier~$1$, and a fixed point at $z=\lambda$ with multiplier $\lambda$.  Since the fixed point multipliers of $\phi$ and $\psi$ coincide,  $\phi(z) \in [\psi]$.

Now assume the three multipliers are distinct, and let 
\begin{equation}\label{multcubic}
f(x) = x^3 - \sigma_1 x^2 + \sigma_2 x - (\sigma_1-2)
\end{equation}
 be the cubic polynomial whose  roots in  $\overline K$ are the multipliers of the fixed points of~$\psi$. A simple calculation shows that if $f(\lambda) = 0$, then $P(\lambda) - \lambda Q(\lambda) = 0$.  In other words, if $\lambda$ is a multiplier for the conjugacy class $[\psi]$, then $z=\lambda$ is a fixed point of the map $\phi(z)$ in~\eqref{phidefalldiffeqn}.   Furthermore, if $f(\lambda) = 0$ and $Q(\lambda) \neq 0$, another calculation shows that $\phi'(\lambda) = \lambda$.  That is, the fixed point $z=\lambda$ has multiplier~$\lambda$.  So as long as the denominator $Q(z)$ does not vanish at any of the three multipliers, $\phi(z) \in K(z)$ has the correct fixed point multipliers.

If $Q(\lambda) = 0$, then
\begin{equation}\label{denomzero}
- \lambda^2 + (2+\sigma_1)\lambda + 2-\sigma_1-\sigma_2 = 0.
\end{equation}
Multiply \eqref{denomzero} by $\lambda$ and add $f(\lambda)=0$ to get
\[
2\lambda^2 + (2-\sigma_1)\lambda-(\sigma_1-2) = 0.
\]
Subtracting \eqref{denomzero}, we obtain
\[
3\lambda^2 - 2\sigma_1 \lambda  +\sigma_2 = 0.
\]
This exactly says that $f'(\lambda) =  0$, so we have a double root at $\lambda$, contradicting our assumption that the
roots of the multiplier cubic are distinct.  Therefore, as long as the roots are distinct, our choice of $Q(z)$ is
never zero at the corresponding fixed points.

As in the previous case, since the fixed point multipliers of $\phi$ and $\psi$ coincide, $\phi(z) \in [\psi]$.

Finally, any other map with the form
\[
\theta(z) = \frac{2 z^2 + (2-\sigma_1') z + (2-\sigma_1')}{-z^2 + (2+\sigma_1')z + 2-\sigma_1'-\sigma_2'}
\]
is conjugate to $\phi(z)$ if and only if $\sigma_1=\sigma_1'$ and $\sigma_2=\sigma_2'$, since these symmetric functions completely determine the conjugacy class.  Hence $\phi(z)$ is the unique map in $[\psi]$ of this form.
\end{proof}

\begin{rem}
The map $\phi$ in equation~\eqref{phidefalldiffeqn} has degree two if and only if the resultant $\Res(P,Q) \neq 0$.  We calculate that
\[
\Res(P,Q) = 
-2\sigma_1^3 - \sigma_1^2\sigma_2 + \sigma_1^2 + 8 \sigma_1\sigma_2 
+ 4\sigma_2^2-12\sigma_1 - 12\sigma +36.
\]
In other words, the vanishing of the resultant corresponds exactly to the symmetry locus given in equation~\eqref{symmlocus}.   Thus we see that the converse of Lemma~\ref{alldif} holds as well.
\end{rem}

\begin{ex}
Since the characteristic of $K$ is not~$3$, the multiplier cubic $f(x)$ in equation~\eqref{multcubic} has distinct roots if it is irreducible.  For instance, if $f(x)=x^3 + 2$, then
$\sigma_1 = \sigma_2 = 0$, so we can take
\[
\psi(z) = \frac{2z^2 + 2z +2}{-z^2 + 2z+ 2}.
\]
\end{ex}

\begin{ex}
We can detect
polynomials easily in this normal form.   If $\phi$ is a polynomial then it is conjugate over $K$ to $z^2 + c$ (recall that we assume the characteristic is not~$2$).  The multiplier at $\infty$ is zero, and since
the other fixed points satisfy $z^2 + c = z$, the other two multipliers sum to $2$ and multiply to $4c$. Hence, the
form presented in Lemma~\ref{alldif} is
\[
\psi(z) = \frac{2z^2}{-z^2 + 4z -4c},
\]
assuming that $c\neq 0$ (that is, excluding the case where $\Aut(\phi) \neq \id$).  This may be helpful in checking the number of preperiodic points over $\qq$,
as we expect different upper bounds for polynomials compared with non-polynomials (see~\cite[Theorem~$2$]{PrePerGraphs} and~\cite[Corollary~$1$]{poonenrefined}).
\end{ex}

This completes part~\eqref{noaut1} of Theorem~\ref{mainthm}.  It remains to consider the cases where the maps in the conjugacy class $[\psi]$ have nontrivial automorphism group.

\section{The case $\Aut(\phi) \cong \mathfrak C_2$}
The following is Lemma~$1$ in~\cite{PrePerGraphs}.
\begin{lem}\label{bkform}
Let $K$ be a  field  with $\ch(K) \neq 2, 3$ and
let $\phi$ be a rational map of degree~$2$ defined over $K$.  Then $\Aut(\phi) \cong \mathfrak C_2$  if and only if  $\phi$ is  conjugate over $K$ to some map of the form
\[
\phi_{k,b}(z) = kz + \frac b z
\]
with $k\in K\smallsetminus\{0,-1/2\}$ and $b \in K^*$.
Furthermore, two such maps $\phi_{k,b}$ and $\phi_{k',b'}$ are  conjugate over $K$ if and only if $k = k'$ and $b/b' \in \left(K^*\right)^2$.  The map $\phi_{k,b}$ has the automorphism $z \mapsto -z$.
\end{lem}

The fixed point multipliers for a map of the form $\phi_{k,b}(z)$ are $\{ 2k-1, 2k-1, 1/k \}$, thus $\phi_{k,b}$ and $\phi_{k',b'}$ are conjugate over $\overline K$ if and only if $k=k'$

Recall that a quadratic rational map $\phi(z)$ has automorphism group $\mathfrak C_2$ if and only if exactly two multipliers are equal and are not~$1$, or if all three multipliers are~$1$.      Hence, this completes part~\eqref{c2aut1} of Theorem~\ref{mainthm}.

\section{The case $\Aut(\phi) \cong \mathfrak S_3$}
As described in Section~\ref{prelim}, for a quadratic rational map $\psi$,  $\Aut(\psi) \cong \mathfrak S_3$ if and only if all three multipliers are~$-2$.  Hence, there is a single $\overline K$-conjugacy class of maps $[\psi]$ with $\Aut(\psi)\cong \mathfrak S_3$, and  $\phi(z) = \frac 1{z^2}\in [\psi]$ since $\Aut(\phi)$ is generated by $z\mapsto 1/z$ and $z \mapsto \omega z$ for $\omega$ a primitive cube root of unity.   We will use $\phi$ as our ``base map'' from which we find all possible $K$-conjugacy classes.

\begin{lem}\label{lem:s3}
Let $\psi(z) \in K(z)$ be a quadratic rational map with $\Aut(\psi) \cong \mathfrak S_3$.  Then:
\begin{enumerate}[\textup(a\textup)]

\item\label{s3irrat2cyc}
The function $\psi(z)$ is conjugate over $K$ to a rational map of the form
\[
\theta_{d,k}(z) = \frac{k z^2 - 2 d z +dk}{z^2 - 2 k z + d},
\qquad \text{with } k \in K,  d\in K^*, \text{ and } k^2\neq d.
\]
  Furthermore, $\theta_{d,k}(z)$ and $\theta_{d',k'}(z)$ are conjugate over $K$ if and only if
  \begin{align*}
  d'&=b^2 d, \text{ and }\\
k' &\in \left\{
\frac {bd}{k},  \frac{b \left(d^2 \gamma ^3+3 d k \gamma ^2+3 d \gamma +k\right)}{d k
   \gamma ^3+3 d \gamma ^2+3 k \gamma +1}.
\right\}
\end{align*}
for some $\gamma \in K$ and $b\in K^*$.
\textup(Note that the choices $\gamma = 0$ and $b=\pm 1$ give $\pm k$ as possibilities for $k'$.\textup)

\item\label{s3rat2cyc}
The function $\psi(z)$ has a $K$-rational two-cycle if and only if it is conjugate over $K$  to
\[
\theta_t(z) = t/z^2.
\]
Furthermore, $\theta_t(z)$ and $\theta_{t'}(z)$ are conjugate over $K$ if and only if $t/t' \in (K^*)^3$ or $t t' \in (K^*)^3$.
\end{enumerate}
\end{lem}

\begin{proof}
 Let $\psi(z) \in K(z)$ be a quadratic rational map defined over $K$ with automorphism group $\mathfrak S_3$.   The proof of~\cite[Theorem~$5.1$]{milnrat} goes through unchanged for fields of characteristic different from $2$ and $3$, so we conclude that $\psi(z)$ is conjugate over~$\overline K$ to the map $\phi(z) = 1/z^2$.  Choose $g\in \PGL_2(\overline K)$ such that $\psi = \phi^g(z)$.  Such a $g$ must take the unique two-cycle $\left(0\leftrightarrow \infty \right)$ of $\phi$ to a (necessarily unique) two-cycle of $\psi$.

 There are polynomials $P_1, Q_1, P_2, Q_2 \in K[z]$ such that $\psi(z)= P_1(z)/Q_1(z)$, and the second iterate $\psi^2(z) = P_2(z)/Q_2(z)$.  Then the second dynatomic polynomial
\[
    \frac{P_2(z) - zQ_2(z)}{P_1(z) - zQ_1(z)}
\]
 is a polynomial in $K[z]$ whose roots are the finite period-$2$ points for~$\psi$.  (For details on dynatomic polynomials, see~\cite[Section4.1]{ads}.)   If this dynatomic polynomial is linear, then $\infty$ is on the two-cycle.  In this case, we may conjugate by some element of $\PGL_2(K)$ which sends the other (necessarily rational) point of period two to~$0$.  Then by an argument identical to the one in~\cite[Theorem~$5.1$]{milnrat}, this conjugacy in fact takes $\psi(z)$ to $\theta_t(z) = t/z^2$ for some $t\in K$.  The rest follows from the proof of part~\eqref{s3rat2cyc} below.
 
 If the second dynatomic polynomial is quadratic, the unique two-cycle for $\psi$ must be of the form $a\pm b\sqrt d$ with $a,b,d \in K$.  Let $f(z) = bz + a \in \PGL_2(K)$.  Then $\psi^f$ has the two-cycle $\left(\sqrt d\leftrightarrow -\sqrt d \right)$, and the conjugacy is over $K$.  So we may assume that $\psi(z)$ has its two-cycle at $\pm \sqrt d$, with $d\in K^*$.

Again, choose $g\in \PGL_2(\overline K)$ such that $\psi = \phi^g(z)$; then based on our knowledge about the two-cycles, it must have the form $g(z) =\frac{\alpha\left(z-\sqrt{d}\right) }{z+\sqrt{d}} $ for some nonzero $\alpha \in \overline K$.  We calculate
\[
    \psi(z) = g^{-1} \circ \phi \circ g (z) =
    \frac{\sqrt d \left(\frac{\alpha^3 + 1}
    {\alpha^3 -1}\right)z^2-2dz
    +d \sqrt d \left(\frac{\alpha^3 + 1}{\alpha^3 -1}\right)}
    {z^2-2 \sqrt d \left(\frac{\alpha^3 + 1}{\alpha^3 -1}\right) z +d}.
\]
Then $\psi(z) \in K(z)$ if and only if $\sqrt d \left(\frac{\alpha^3 + 1}{\alpha^3 -1}\right) =: k \in K$, which leads to the map
\[
    \theta_{d,k}(z) = \frac{P(z)}{Q(z)}=\frac{k z^2-2 d z+d k}{z^2-2 k z+d}.
\]
Since $\alpha \neq 0, \infty$,  $k\neq \pm \sqrt d$.  (And in fact, $\theta_{d,\,\pm \sqrt d}$ are not quadratic maps; the resultant of $P(z)$ and $Q(z)$ vanishes precisely when $d \in \{ 0, k^2 \}$.)

Since every $\theta_{d,k}$ is conjugate to $\phi(z) = 1/z^2$, each of these maps has automorphism group $\mathfrak S_3$.  We have now proved that $\psi$ is conjugate over $K$ to some $\theta_{d,k}$, so it remains only to decide when two such maps are conjugate to each other over $K$.

First, if two maps are conjugate over $K$, their two-cycles must generate the same field extension of $K$, and $\sqrt d$ and $\sqrt {d'}$ generate the same extension if and only if $d'/d \in \left(K^*\right)^2$.
Hence, if  $\theta_{d, k}$ is conjugate over $K$ to $\theta_{d',k'}$ we must have  $d'=b^2 d$ for some $b \in K^*$.

Now assume that $h \in \PGL_2(K)$ satisfies $\theta_{d,k}^h = \theta_{b^2 d,k'}$.  Because $h$ must send the two-cycle of $\theta_{b^2 d,k} $ to the two-cycle of $\theta_{d,k} $, we can choose the sign of $b$ so that  $h(b \sqrt d) = \sqrt d$ and $h(-b\sqrt d) = - \sqrt d$.  Therefore
\[
    h = \begin{pmatrix} \beta & -b d \gamma \\-\gamma & b \beta \end{pmatrix}.
\]
If $\beta = 0$, we conjugate $\theta_{d,k}$ by $h$ to find that $k' = b d/k$.

If $\beta \neq 0$, we may take $h$ to be
\[
    h = \begin{pmatrix} 1 & -b d \gamma \\ -\gamma & b \end{pmatrix}.
\]
In this case, we conjugate $\theta_{d,k}$ by $h$ to find that
\[
k' = \frac{b \left(d^2 \gamma ^3+3 d k \gamma ^2+3 d \gamma +k\right)}{d k
   \gamma ^3+3 d \gamma ^2+3 k \gamma +1}.
\]

For part (b), if the two-cycle of $\psi$ is rational, we may certainly conjugate over $\PGL_2(K)$ so that the two-cycle is $(0 \leftrightarrow \infty)$.

If both period-$2$ points of $\psi$ are finite, then it is clear from above that the unique two-cycle of $\psi$ is defined over $K$ if and only if $\psi$ is conjugate over $K$ to $\theta_{d,k}$ with $d\in (K^*)^2$.  In this case, conjugation by $\frac{z-\sqrt d}{z+\sqrt d} \in \PGL_2(K)$ takes the two-cycle $(\sqrt d \leftrightarrow -\sqrt d)$ of $\theta_{d,k}$ to $(0\leftrightarrow \infty)$, and the resulting map is
\[
\theta_t = \frac t{z^2}, \quad \text{ where } t = \frac{-\sqrt d + k}{\sqrt d +k} \in K^*.
\]
(Note that if one of the period-$2$ points is $\infty$, we can now conjugate $\theta_t(z) = t/z^2$ into the form given in part~\eqref{s3irrat2cyc}.)

Now suppose we have $h\in \PGL_2(K)$ such that $\theta_t^h = \theta_{t'}$.  Since $\{0,\infty\}$ is fixed by $h$, either $h(z) = \lambda z$ or $h(z) = \lambda/z$ for some $\lambda \in K^*$.  Conjugating $\theta_t(z)$ by $h(z) = \lambda z$ and solving leads to $t/t' = \lambda^3$.  Conjugating by $h(z) = \lambda/z$ leads to $t t' = \lambda^3$.
\end{proof}

This completes part~\eqref{s3aut} of Theorem~\ref{mainthm}, and the proof is now complete.

\begin{rem}
It would, of course, be desirable to have the conjugacy condition in Lemma~\ref{lem:s3} be more obviously symmetric in $k$ and $k'$ and easier to check for any two given quadratic rational maps.

In the case $K=\qq$, we may alter Lemma~\ref{lem:s3} to take $d\in \qq^*$ squarefree, which then forces $b=\pm1$.  Furthermore, in this case a calculation shows the following:

\begin{align*}
k'=d/k &\Leftrightarrow
\frac{\left(k + \sqrt{d}\right)
   \left(k'-\sqrt{d}\right)}{\left(k-\sqrt{d}\right)
   \left(k'+\sqrt{d}\right)} = -1,\\
k'=-d/k &\Leftrightarrow
\frac{\left(k + \sqrt{d}\right)
   \left(k'+\sqrt{d}\right)}{\left(k-\sqrt{d}\right)
   \left(k'-\sqrt{d}\right)} = -1,\\
 k'= \frac{ d^2 \gamma ^3+3 d k \gamma ^2+3 d \gamma +k}{d k
   \gamma ^3+3 d \gamma ^2+3 k \gamma +1}
   &\Leftrightarrow
\frac{\left(k+\sqrt{d}\right)
   \left(k'-\sqrt{d}\right)}
   {\left(k-\sqrt{d}\right)
   \left(k'+\sqrt{d}\right)} = 
      \left(\frac{1-\gamma\sqrt{d}}{1+ \gamma\sqrt{d} } \right)^3, \text{ and}\\
  k'= -\frac{ d^2 \gamma ^3+3 d k \gamma ^2+3 d \gamma +k}{d k
   \gamma ^3+3 d \gamma ^2+3 k \gamma +1}
   &\Leftrightarrow
\frac{\left(k+\sqrt{d}\right)
   \left(k'+\sqrt{d}\right)}
   {\left(k-\sqrt{d}\right)
   \left(k'-\sqrt{d}\right)} = 
      \left(\frac{1-\gamma\sqrt{d}}{1+ \gamma\sqrt{d} } \right)^3.
  \end{align*}
  
Now suppose
\[
	\alpha = \left(\frac{1-\gamma\sqrt{d}}{1+ \gamma\sqrt{d} } \right).
\]

If $\alpha \neq  -1$, we may rearrange to find
\[
 \gamma \sqrt d = \frac{1-\alpha}{1+\alpha}.
\]

Since $\gamma \in \qq$, this puts restrictions on $\alpha$.  Namely, we rationalize the denominator to find

\[
	\gamma \N(1+\alpha) \sqrt d = (1-\alpha)(1+\overline\alpha),
\]
where $\N$ is the usual norm in the quadratic number field $\qq(\sqrt d)$.  Since $\gamma \N(1+\alpha) \in \qq$, we conclude that $1 - \alpha \overline\alpha = 0$, or in other words $\N(\alpha) =1$.

Conversely, if $\N(\alpha) = 1$, then $(1-\alpha) (1+ \overline\alpha) = \xi \sqrt d$ for some $\xi \in \qq$.  Note that if $\alpha \neq -1$, $N(1+\alpha) \neq 0$, so we may set $\gamma = \xi / \N(1+\alpha)$ and reverse the argument above to get 
\[
	\alpha = \left(\frac{1-\gamma\sqrt{d}}{1+ \gamma\sqrt{d} } \right).
\]

In other words, for $K=\qq$ we have the following simplified version of Lemma~\ref{lem:s3} part~\eqref{s3irrat2cyc}: The function $\psi(z)$ is conjugate over $\qq$ to a rational map of the form
\[
\theta_{d,k}(z) = \frac{k z^2 - 2 d z +dk}{z^2 - 2 k z + d},
\qquad \text{with } k \in K \text{ and } d \text{ squarefree.}
\]
  Furthermore, $\theta_{d,k}(z)$ and $\theta_{d',k'}(z)$ are conjugate over $K$ if and only if
  $d'= d$, and one of
\[
  \frac{\left(k+\sqrt{d}\right)
   \left(k'-\sqrt{d}\right)}{\left(k-\sqrt{d}\right)
   \left(k'+\sqrt{d}\right)}
   \qquad
   \text{ or } 
   \qquad
  \frac{\left(k+\sqrt{d}\right)
   \left(k'+\sqrt{d}\right)}{\left(k-\sqrt{d}\right)
   \left(k'-\sqrt{d}\right)} 
\]
is a perfect cube of norm~1 in the field $\qq(\sqrt d)$.  

\end{rem}

\bibliographystyle{plain}
\bibliography{refs}

\end{document}